\documentclass[11pt]{amsart}

\usepackage{amsmath, amsthm, amssymb}

\usepackage[T1]{fontenc}

\usepackage{palatino}         
\usepackage{comment}

\usepackage[abs]{overpic}		  

\usepackage{xcolor}
\definecolor{indigo}{rgb}{0.29, 0.0, 0.51}  
\usepackage[colorlinks, urlcolor=indigo, linkcolor=indigo, citecolor=indigo]{hyperref}

\theoremstyle{plain}
\newtheorem{theorem}{Theorem}
\newtheorem{corollary}[theorem]{Corollary}

\newtheorem{lemma}[theorem]{Lemma}
\newtheorem{question}[theorem]{Question}
\newtheorem{conjecture}[theorem]{Conjecture}

\theoremstyle{definition}

\theoremstyle{remark}
\newtheorem{remark}[theorem]{Remark}
\newtheorem{example}[theorem]{Example}

\numberwithin{theorem}{section}

\newcommand{\dfn}[1]{{\em #1}}        
\newcommand{\R}{\mathbb{R}}           
\newcommand{\Z}{\mathbb{Z}}           
\newcommand{\Q}{\mathbb{Q}}           
\newcommand{\N}{\mathbb{N}}           

\makeatletter
\newcommand*\bigcdot{\mathpalette\bigcdot@{0.6}}
\newcommand*\bigcdot@[2]{\mathbin{\vcenter{\hbox{\scalebox{#2}{$\m@th#1\bullet$}}}}}
\makeatother

\DeclareMathOperator\tb{tb}                   
\DeclareMathOperator\rot{rot}                 
\DeclareMathOperator\self{sl}                 

\begin{document}

\title{Neighborhoods of transverse knots and destabilizations} 

\author{John B. Etnyre}

\address{School of Mathematics \\ Georgia Institute of Technology \\  Atlanta, GA}
\email{etnyre@math.gatech.edu}


\begin{abstract}
In this note, we show that transverse knots have unique standard neighborhoods and prove a structure theorem about non-loose Legendrian knots. We also prove a finiteness result for transverse knots in a tight contact manifold. 
The common theme of these two results is a general destabilization result for Legendrian knots. As a byproduct of this work, we find a manifold with an infinite number of distinct tight contact structures, up to contactomoprhism, with no Giroux torsion. 
\end{abstract}

\maketitle

\section{Introduction}

Understanding neighborhoods of Legendrian and transverse knots and when they destabilize is essential in contact geometry. For example, a key step in most classification results for Legendrian and transverse knots is showing when they destabilize \cite{EtnyreHonda01b}, and insight into neighborhoods of knots is necessary in the study of Legendrian and transverse representatives of cable and satellite knots \cite{EtnyreHonda05, EtnyreLafountainTosun12, EtnyreVertesi18}, various constructions (such as Lutz twists \cite{Geiges08, Lutz77}), and the classification of contact structures on surgered manifolds \cite{EtnyreMinTosunVarvarezosPre}. 

In this paper, we will give conditions under which a Legendrian knot in a tight contact structure will destabilize. We also give a criterion for non-loose Legendrian knots in an overtwisted contact structure to ``destablize to infinity''. We also show that there are only finitely many transverse knots in a tight contact manifold with a fixed self-linking number and smooth knot type. We then turn to the question of when a transverse knot has a unique standard neighborhood, and show that sufficiently negative neighborhoods are unique, but in general, standard neighborhoods are not. Finally, out of this study, we show that there are manifolds with boundary that support an infinite number of distinct tight contact structures, up to contactomoprhism, with no Giroux torsion. This is somewhat unexpected, given that \cite{ColinGirouxHonda09} shows this cannot happen for manifolds with convex boundary. 


\subsection{Destablizing Legendrian knots}\label{destabsection}

Determining when a Legendrian knot destabilizes is a key to classifying Legendrian representatives of a given smooth knot type \cite{EtnyreHonda01b}, it can also be useful in understanding fillings of a contact manifold obtained by surgery on a Legendrian knot \cite{ChristianMenkePre}. The following conjecture would go a long way towards a general understanding of Legendrian knots. 
\begin{conjecture}\label{co$nj}
Given a null-homologous smooth knot type $K$ in a tight contact manifold $(M,\xi)$, there is an integer $n$ such that any Legendrian representative of $K$ with Thurston-Bennequin invariant less than $n$ destabilizes. 
\end{conjecture}
If this conjecture were true, then we would know that Legendrian representatives of $K$ are ``finitely generated''. By this, we mean that there would be a finite number of distinct non-destabilized Legendrian representatives of $K$, and all other Legendrian representatives are stabilizations of these. This follows from the fact that there are finitely many Legendrian representatives of $K$ with a fixed Thurston-Bennequin invariant and rotation number \cite{ColinGirouxHonda09}, and given the conjecture and the Bennequin inequality, there would only be a finite number of possible Thurston-Bennequin invariants and rotation numbers that could be realized by non-destabilizable Legendrian knots. 

While we cannot prove the above conjecture, we make partial progress by giving a general criterion that implies a Legendrian knot must destabilize. 
\begin{theorem}\label{maindestab}
Let $L$ be a Legendrian knot in a contact manifold $(M,\xi)$. If $\Sigma$ is a surface with $\partial \Sigma=L$, $\tb(L)<-1$, and $\pm \rot(L)<\chi(\Sigma)$, then $\Sigma$ may be isotoped, rel $L$, so that there is a $\pm$-bypass for $L$ on $\Sigma$. In particular, $L$ destabilizes. 
\end{theorem}
We will use this proposition to prove the main results about neighborhoods of transverse knots discussed below.

\subsection{Destabilizing non-loose Legendrian knots}
Recall a Legendrian knot $L$ in an overtwisted contact manifold is called \dfn{loose} if the contact structure restricted to the complement of a standard neighborhood of $L$ is also overtwisted. Otherwise, we call the knot $L$ \dfn{non-loose}. Loose knots are fairly well understood, see \cite{Etnyre13}, but we currently know very few general results about the structure of non-loose Legendrian knots in overtwisted contact manifolds. One such result is a generalization of the Bennequin bounds. Specifically, if $L$ is a non-loose null-homologous knot in an overtwisted contact manifold then 
\[
-|\tb(L)|+|\rot(L)|\leq -\chi(L),
\]
where $\tb(L)$ is the Thurston-Bennequin invariant of $L$, $\rot(L)$ is the rotation number of $L$, and $\chi(L)$ is the maximal Euler characteristic of a surface with boundary $L$. See \cite{Swiatkowski92}.

In all known examples where we have a complete classification of non-loose knots, see \cite{EliashbergFraser09, EtnyreMinMukherjee22Pre}, we always see that some non-loose representatives ``destabilize to infinity'', that is, there are non-loose knots that can destabilize infinitely often. So we ask the following question.
\begin{question}
If $K$ is a null-homologous knot type in a manifold $M$ and $K$ admits non-loose Legendrian representatives in a contact structure $\xi$ on $M$, then does some non-loose Legendrian representative of $K$ destabilize infinitely many times? 
\end{question}

While we cannot currently answer this question, we can give strong evidence that the answer is ``yes'' with the following theorem that gives a criterion for a non-loose knot to destabilize to infinity. 

\begin{theorem}\label{mainnonloose}
Let $L$ be a non-loose null-homologous Legendrian knot in an overtwisted contact manifold $(M,\xi)$. If $\pm \rot(L)<\chi(L)$ and $\tb(L)>0$, then there are non-loose knots $L_i$, for $i\in \N$, such that $L_0=L$ and $L_{i-1}=S_\mp(L_i)$. 
\end{theorem}
Pictorially, this theorem says that if there is a non-loose Legendrian knot with rotation number and Thurston-Bennequin invariant in the shaded region of Figure~\ref{destab}, 
\begin{figure}[htb]{
\begin{overpic}
{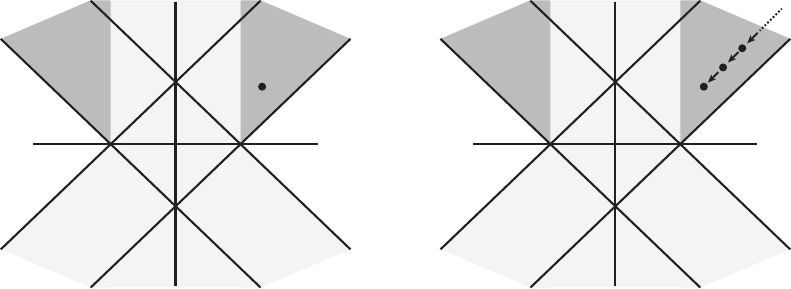}
\end{overpic}}
\caption{The shaded region on the left is the region satisfying $-|\tb(L)|+|\rot(L)|\leq -\chi(L)$. The darker shaded region in which Theorem~\ref{mainnonloose} implies that Legendrian knots must destabilize. Specifically, if a Legendrian knot has invariants as indicated by the dot on the left-hand side diagram, then it negatively destabilizes to infinity as indicated in the diagram on the right-hand side. If the dot is in the left darkly shaded region, then it will positively destabilize to infinity; that is, we see the right-hand diagram reflected about the vertical axis.}
\label{destab}
\end{figure}
then one must also have non-loose representatives as shown on the right of Figure~\ref{destab}.

\subsection{Transverse knots with fixed self-linking number}
Using our understanding of when Legendrian knots in tight contact destabilize, we can establish the transverse analog of the result of Colin, Giroux, and Honda mentioned above, that there are finitely many Legendrian knots with a fixed Thurston-Bennequin invariant and rotation number. 
\begin{theorem}\label{detablizetransverse}
Given a null-homologous smooth knot type $K$ in a tight contact manifold $(M,\xi)$, there are only finitely many transverse representatives of $K$ with a fixed self-linking number. 
\end{theorem}
Surprisingly, the analogous result for Legendrian knots has been known for over 17 years due to the work of Colin, Giroux, and Honda in \cite{ColinGirouxHonda09}, while the transverse version has remained open. Shortly after that work, the transverse finiteness result was first considered in the thesis of Guyard \cite{Guyard15}, where it was reduced to a statement about non-destablizable Legendrian knots, and then some progress on that statement was made. Our approach to this theorem is to reduce it to the weaker statement about destabilizing Legendrian knots that is established in Theorem~\ref{maindestab}.

\subsection{Neighborhoods of transverse knots}
We begin by discussing neighborhoods of transverse knots. We recall that a transverse knot $T$ in a contact $3$-manifold $(M,\xi)$ has a neighborhood $N_s$ that is contactomoprhic to a standard model $\widehat N_s$ for some $s\in \Q$. We will discuss the standard model in Section~\ref{nbhds}, but $\widehat N_s$ is a solid torus with a fixed contact structure, and the characteristic foliation on $\partial \widehat N_s$ is linear with slope $s$. Moreover, if $T$ has a neighborhood $N_s$ then it has a neighborhood $N_{s'}$ for any $s'\leq s$.\footnote{The issue of framing on $T$ is important when discussing slopes. We discuss this more thoroughly in Section~\ref{nbhds}, but for now we note that we fix a framing of $T$ and use it whenever discussing slopes on the boundary of a neighborhood of $T$.} A basic question is whether or not $N_s$ is well-defined up to ambient contact isotopy. Clearly, two neighborhoods $N_s$ and $N_{s'}$ of $T$ cannot be contact isotopic if $s\not=s'$ as the characteristic foliation is preserved under contact isotopy. So we frame this basic question as follows: if $T$ and $T'$ are transversely isotopic transverse knots and $N_s$ and $N'_s$ are their standard neighborhoods, then is there a contact isotopy of $(M,\xi)$ that takes $N_s$ to $N'_s$?

We note that the question asked above is not only a basic question, but it is fundamental to many constructions. For example, when one defines a Lutz twist on $T$, \cite{Geiges08}, one starts with a standard neighborhood $N_s$ of $T$, removes it and glues in a new solid torus with another contact structure. So if $N_s$ is not well-defined, then a Lutz twist depends on the choice of $N_s$ and not just on $T$. 
As Lutz twists are a fundamental construction in contact geometry, it would be helpful to know if they only depend on the transverse knot $T$ or on the choice of the neighborhood of $T$. 

Similarly, in \cite{Gay02a}, admissible transverse surgery on a transverse knot was defined. It again starts with a neighborhood $N_s$ of a transverse knot $T$ and performs surgery on $N_s$. So again, to see that admissible transverse surgery is well-defined, we need to see that $N_s$ depends only on $T$ (and $s$). 

In the literature, it seems that the well-definiteness of $N_s$ has largely been assumed \cite{Gay02a, Geiges08}, but some experts have known for some time that $N_s$ is not well-defined for all $T$ and all $s$. We discuss this more below, but we first note that as long as $s$ is sufficiently negative, we do have a well-defined neighborhood of $T$.

\begin{theorem}\label{nbhdoftransverse}
Let $(M,\xi)$ be a tight contact $3$-manifold. 
Given a null-homologous knot type $K$ in a manifold $M$ and an integer $n$, there is a rational number $r$ such that any standard neighborhood with boundary slope $s<r$ of any transverse knot $T$ in the knot type $K$ with $\self(T)\geq n$ is unique up to contact isotopy. 
\end{theorem}
\begin{remark}
Our proof will show that the same theorem is true in overtwisted contact manifolds as long as we only consider neighborhoods of non-loose knots. Surprisingly, it is not clear if the result is true for loose knots in overtisted contact manifolds (thought it is coarsely true, in the sense that the neighborhoods will be unique up to a contactomorphism that is smoothly isotopic to the identity). 
\end{remark}

This theorem is sufficient to see that Lutz twists on $T$ only depend on $T$, admissible surgery on $T$ only depends on $T$ for sufficiently negative surgery slopes, and inadmissible surgery, \cite{Conway19}, is well-defined for any surgery slope. 

We can estimate $r$ in the theorem above in some cases. To state this estimate we recall that a knot type $K$ is called \dfn{Legendrian simple} if Legendrian knots smoothly isotopic to $K$ are determined by their Thurston-Bennequin invariant and rotation number. 
\begin{theorem}\label{main2}
If $K$ is a Legendrian simple knot in a tight contact manifold $(M,\xi)$ and $K=\partial \Sigma$ for some embedded surface $\Sigma$, then a standard neighborhood $N_s$ of a transverse knot $T$ in the knot type $K$ is unique up to contact isotopy if $s<\self(T)+\chi(\Sigma)$, where $\chi(\Sigma)$ is the Euler characteristic of $\Sigma$ and $\self(T)$ is the self-linking number of $T$. 
\end{theorem}

We would now like to address when $N_s$ is well-defined and when it is not, as well as the range of $s$ for which a given transverse knot $T$ has a neighborhood $N_s$. To this end, we recall that a smooth knot type $K$ is \dfn{uniformly thick} if any solid torus $S$ in the knot type $K$ is contained in a solid torus $S'$ that is a standard neighborhood of a Legendrian knot with maximal Thurston-Bennequin invariant.

\begin{theorem}\label{main3}
If $K$ is Legendrian simple and uniformly thick in a tight contact manifold, then any standard neighborhood of a transverse representative of $K$ is unique up to contact isotopy. 
\end{theorem}
In this situation, we can also determine for which $s$ we can find a standard neighborhood $N_s$ of $T$ in terms of the classification of Legendrian knots isotopic to $K$. In Remark~\ref{thealgorithm}, we will give an algorithm to determine these values of $s$, but we state here a few examples. 
\begin{theorem}\label{ex1}
Consider the $(2,-2n-1)$-torus knot $K$ in $(S^3,\xi_{std})$, where $\xi_{std}$ is the unique tight contact structure on $S^3$. The maximal self-linking number of $K$ is $-2n-3$. Let $T$ be a transverse representative of $K$. If 
\[
\self(K)\in\{-2n-3, -2n-5, \ldots,-6n-1\},
\] 
then $T$ has a standard neighborhood $N_s$ if and only if $s<-4n-2$ and if 
\[
\self(T)=-2k-1 \text{ for } k>3n,
\] 
then $T$ has a standard neighborhood $N_s$ if and only if $s<-k-n-2$. 

Moreover, any standard neighborhood of any transverse representative of $K$ is determined up to contact isotopy by the slope on its boundary. 
\end{theorem}
In this example, note that for self-linking numbers near the maximal possible value, the transverse knots have standard neighborhoods with the same possible slopes, but after some point, the possible slopes decrease as the self-linking number decreases. We note that this is not always the case!
\begin{theorem}\label{ex2}
Consider the $(4,-9)$-torus knot $K$ in $(S^3,\xi_{std})$. The maximal self-linking number of $K$ is $-31$. Let $T$ be a transverse representative of $K$. If 
\[
\self(T)\in \{-31, -33, -39, -41\}
\]
then $T$ has a standard neighborhood $N_s$ if and only if $s<-36$, if
\[
\self(T)\in \{-35, -37\}
\]
then $T$ has a standard neighborhood $N_s$ if and only if $s<-37$, and if 
\[
\self(T)=-2k-1
\]
for $k>20$, then $T$ has a standard neighborhood $N_s$ if and only if $s<-k-16.$

Moreover, any standard neighborhood of any transverse representative of $K$ is determined up to contact isotopy by the slope on its boundary. 
\end{theorem}
Though not uniformly thick, we can also determine the standard neighborhoods of transverse unkots. 
\begin{theorem}\label{ex3}
Let $T$ be a transverse unknot in $(S^3,\xi_{std})$. If $\self(T)=-1$ then $T$ has a standard neighborhood $N_s$ if and only if $s<0$ and if $\self(T)=-2k-1$ for $k>0$, then $T$ has a standard neighborhood $N_s$ if and only if $s<-k-1.$

Moreover, any standard neighborhood of any transverse representative of $K$ is determined up to contact isotopy by the slope on its boundary. 
\end{theorem}
We now turn to transverse knots that have non-unique standard neighborhoods. In Remark~\ref{remark} we will see that all positive torus knots have transverse representatives with non-unique standard neighborhoods, but here, we give one example of this phenomenon.
\begin{theorem}\label{ex4}
Let $K$ be the right-handed trefoil in $(S^3,\xi_{std})$. A transverse representative of $K$ with self-linking number $-2k-1$, for $k\geq 0$, has a standard neighborhood $N_s$ if and only if $s<-k$. Moreover, any standard neighborhood of any such transverse representative of $K$ is determined up to contact isotopy by the slope on its boundary. 

If $T$ is the unique transverse knot representing $K$ with $\self(T)=1$, then $T$ has a standard neighborhood $N_s$ if and only if $s<1$. The neighborhood $N_s$ is unique up to contact isotopy if and only if $s<0$ or $s\in[1/2,1)$. If $s\in[1/(n+1), 1/n)$, then $T$ has $n$ distinct standard neighborhoods with boundary slope $s$. And if $s=0$, then there are infinitely many distinct standard neighborhoods. 
\end{theorem}

Given the examples we know of transverse knots with non-unique neighborhoods, we make the following conjectures. 

\begin{conjecture}
For any transverse knot $T$ in any contact manifold, any standard neighborhood of $T$ with boundary slope less than zero is unique up to contact isotopy. 
\end{conjecture}

\begin{conjecture}
In a fixed knot type $K$, there is some integer $n$ such that if $T$ is a transverse representative of $K$ with $\self(T)<n$, then any standard neighborhood of $T$ is determined up to contact isotopy by its boundary slope. 
\end{conjecture}

The proof of this Theorem~\ref{nbhdoftransverse}  will depend on Legendrian approximations of a transverse knot, see Section~\ref{approx} for details on Legendrian approximations. The well-definedness of Legendrian approximations has also not been addressed in the literature. To prove the above theorem, we will first consider this issue. 
\begin{lemma}\label{part2}
Let $K$ be a null-homologous knot type in a contact manifold $(M,\xi)$. Assume that either $\xi$ is tight or only consider non-loose representatives of $K$. Given any integer $n$, there is an integer $m$ such that any two Legendrian approximations of a transverse knot $T$ in the knot type $K$ with $\self(T)>n$ with the same Thurston-Bennequin invariant less than $m$ are Legendrian isotopic. 
\end{lemma}
We now briefly explore examples of knot types where the Legendrian approximation of any transverse knot is determined by its Thurston-Bennequin invariant and other examples where this is not true. 
\begin{example}
If $K$ is any Legendrian simple knot, then clearly, Legendrian approximations of transverse knots in this knot type are determined by their Thurston-Bennequin invariant, and the possible Thurston-Bennequin invariants are determined by the Legendrian mountain range of $K$. For example, if $T_{-2k-1}$ is a transverse unknot with self-linking number $-2k-1$, then it may be approximated by a Legendrian unknot with $\tb$ any integer less than or equal to $-k$. We have similar results for torus knots and the figure eight knot.
\end{example}

\begin{example}
Now consider the twist knot $T_{-2n-1}$. The classification of Legendrian twist knots can be found in \cite{EtnyreNgVertesi13}. If $T$ is a transverse representative of $T_{-2n-1}$ with self-linking number $-2k-3$ for $k\geq 0$, then if $k>0$ its Legendrian approximations are determined by their Thurston-Bennequin invariants, which can be any integer less than or equal to $-2-k$. When $k=0$, then Legendrian approximations with $\tb<-3$ are determined by their Thurston-Bennequin invariant, which can be any integer less than or equal to $-4$. However, there are $n$ distinct Legendrian approximations of $T$ with $\tb=-3$. 

Negative even twist knots $T_{-2n}$ are more complicated. Here we will focus on $K_{-6}$. Using the classification in \cite{EtnyreNgVertesi13}, other values of $n$ can easily be worked out. We first recall that there are exactly $2$ transverse knots $T$ and $T'$ in the knot type $K_{-6}$ with self-linking number $1$. And exactly one transverse representative $T_{-1-2k}$ with $\self=-1-2k$, for $k\geq 0$. Each $T_{-1-2k}$ has a unique Legendrian approximation with $\tb\leq -k-1$, but there are two Legendrian approximations of $T_{-1-2k}$ with $\tb=-k$. Similarly, $T$ and $T'$ each have unique Legendrian appoximations with $\tb\leq 0$, but one of them, say $T$, has $3$ Legendrian appoximations with $\tb=1$, while $T'$ has $2$ Legendrian appoximations with $\tb=1$. 
\end{example}

\subsection{Infinitely many distinct contact structures}
The proof of Theorem~\ref{ex4} has a curious corollary. 
\begin{theorem}\label{curous}
Let $M$ be the complement of the right-handed trefoil in $S^3$. There are infinitely many distinct, up to contactomorphism, contact structures on $M$ with no Giroux torsion that induce a linear foliation on $\partial M$ of slope $0$. 
\end{theorem}
We note that by work of Colin, Giroux, and Honda \cite{ColinGirouxHonda09} an atoroidal manifold only admits finitely many tight contact structures up to isotopy. This is also true if the manifold has a convex boundary. The main way to create infinitely many contact structures up to isotopy is by introducing Giroux torsion along an incompressible torus. If one is considering contact structures up to isotopy, then one can construct infinitely many distinct tight contact structures (for example, on circle bundles and Seifert fibered spaces), but the examples in the theorem above are, to the author's knowledge, the first infinite family of contact structures, up to contactomorphis, with zero Giroux torsion.

\subsection*{Acknowledgements}
The author thanks Kai Cieliebak for helpful conversations about Lutz twists that led to the realization that neighborhoods of transverse knots had not previously been understood up to contact isotopy. We also thank Hyunki Min for helpful conversations and Bülent Tosun for helpful conversations and encouragement to write up these results. We also greatful for discussions with Marc Kegel and Isacco Nonino that led to including Theorem~\ref{detablizetransverse} in the paper. We finally thank Jeremy Van Horn-Morris and Shea Vela-Vick for helpful discussions years ago that led to the proof of Proposition~\ref{maindestab}. The author was partially supported by the National Science Foundation grant DMS-2203312 and the Georgia Institute of Technology's Elaine M. Hubbard Distinguished Faculty Award.

\section{Background}\label{background}
We assume the reader is familiar with the basic results from contact geometry and the theory of Legendrian and transverse knots as can be found in \cite{Etnyre05}, see also \cite{EtnyreTosunPre24, Geiges08}. We recall a few results here that are needed below. We will also use standard facts about convex surfaces. Most of these facts can be found in \cite{EtnyreHonda01b, Honda00a}, but we caution that the conventions for slopes of curves on tori and the Farey graph are different than those that we use here. For the ``modern'' conventions, we refer the reader to \cite{EtnyreRoy21, EtnyreTosunPre24}

\subsection{The classical transverse knots neighborhood theorem}\label{nbhds}
Consider the manifold $\R^2\times S^1$ with the contact structure $\ker(\cos r\, d\phi + r\sin r\, d\theta)$,  where $(r,\theta)$ are polar coordinates on $\R^2$ and $\phi$ is an angular coordinate on $S^1$. Note that the torus $T_a$ of radius $a$, that is $\{(r,\theta,\phi): r=a\}$, has a linear characteristic foliation with slope $-\frac 1a \cot a$. Given a number $s$ there will be a unique smallest positive $a$ such that $s=-\frac 1a \cot a$. We define the solid torus $\widehat N_s$ to be $\{(r,\theta,\phi): r\leq a\}$. So $\partial \widehat N_s$ has linear characteristic foliation of slope $s$ and given any $s'<s$, the torus $\widehat N_{s'}$ is contained in $\widehat N_s$. (Note as $s$ goes to $-\infty$, the corresponding values of $a$ go to $0$.)

Now given a transverse knot $T$ and a choice of framing on $T$, it is well-known, see \cite{EtnyreTosunPre24, Geiges08}, that $T$ has a neighborhood $N_s$ that is contactomorphic to $\widehat N_s$ for sufficiently small $s$ via a contactomorphism taking the framing to the product framing on $\widehat N_s$. When $T$ is null-homologous, we will always give $T$ the framing coming from a surface that it bounds. 
We call $N_s$ a \dfn{standard neighborhood} of $T$.  From the discussion above, if $T$ has such a neighborhood, then it also has a neighborhood $N_{s'}$ for any $s'\leq s$. One main goal of this paper is to show that $N_s$ only depends on the transverse isotopy class of $T$ and $s$ up to ambient contact isotopy if $s$ is sufficiently negative and we will see examples where there are distinct $N_s$ with $s$ is large.


\subsection{Transverse push-offs and Legendrian approximations}\label{approx}
Throughout this paper, we will need to go between transverse and Legendrian knots, so we review that process here.

We first recall that a Legendrian knot has a unique transverse push-off. More specifically, given an oriented Legendrian knot $L$ in a contact manifold $(M,\xi)$ it is well known that $L$ has a standard neighborhood $N_L$ contactomorphic to $S^1\times D_\epsilon=\{(\phi, y,z): y^2+z^2 \leq \epsilon^2\}$ with the contact structure $\ker(dz -y\, d\phi)$, \cite{EtnyreTosunPre24, Geiges08}, where $\phi$ is the angular coordinate on $S^1$ and $(y,z)$ are Cartisian coordinates on the disk and the orientation on $L$ is identified with the positive $\phi$ direction. We can now define the \dfn{transverse push-off of $L$}, denoted $L^t$, to be the image of $S^1\times\{(0,\epsilon')\}$ for any $\epsilon'\in (0,\epsilon)$. One may check that the transverse isotopy class of $L^t$ depends only on the Legendrian isotopy class of $L$, that 
\[
\self(L^t)=\tb(L)-\rot(L),
\]
and that $(S_-(L))^t=L^t$ where $S_-(L)$ is the negative stablization of $L$ and $(S_+(L))^t$ is the transverse stablization of $L^t$ where $S_+(L)$ is the positive stablization of $L$. See \cite{EtnyreHonda01b}. 

We also notice that the classification of tight contact structures on solid tori \cite{Giroux00, Honda00a} immediately implies the following lemma we will need below. 
\begin{lemma}\label{subtori}
With the notation above, inside of $N_L$ we can find a solid torus $N_s$ that will be a standard neighborhood of $T$ for any $s<\tb(K)$. 
\end{lemma}

Now, given a transverse knot $T$ in a contact manifold, we can find Legendrian knots that have $T$ as their transverse push-off. To see this, from the previous subsection, we know $T$ has a standard neighborhood $N_s$ for some $s<0$. Now for any $n<s$ we can consider the sub-neighborhood $N_n$ inside of $N_\epsilon$. Let $L_n$ be a leaf of the characteristic foliation of $\partial N_n$. Clearly $L_n$ is an $(1,n)$-curve on $\partial N_n$ and hence smoothly isotopic to $T$. It is not hard to see that the transverse push-off $L_n^t$ is transverse isotopic to $T$ and that $L_{n-1}=S_-(L_n)$, \cite{EtnyreHonda01b}. So there is not a single Legendrian knot related to $T$, but a series of them all related by negative stabilization. We call $L_n$ for any $n$ a \dfn{Legendrian approximation to $T$}. 

We finally recall that given a smooth knot type $K$, the classification of Legendrian knots in the knot type $K$ up to Legendrian isotopy and negative stabilization is equivalent to the classification of transverse knots in a smooth knot type $K$; moreover, the identification is given by transverse push-off and Legendrian approximations. See \cite{EpsteinFuchsMeyer01, EtnyreHonda01b}.

We recall how to explicitly use this last result to understand transverse knots in a smooth knot type $K$ in terms of Legendrian realizations of $K$. Suppose we have a classification of Legendrian representatives of $K$. We first consider the subset $\mathcal{M}_K$ of $\Z^2$ consisting of ordered pairs $(\rot(L),\tb(L))$ for all Legendrian realizations of $K$ and two each point in $\mathcal{M}_K$ we associate the number of Legendrian isotopy classes of Legendrian realizations of $K$ with those invariants. We call $\mathcal{M}_K$ the \dfn{mountain range of $K$}. (To completely understand the classification of Legendrian knots in the knot type $K$, we should also know what happens when Legendrian knots with invariants $(n,m)$ are stabilized positively and negatively.) Now, to understand the transverse representatives of $K$ with $\self=n$ we consider the intersection of the line of $t=n+r$ in $\Z^2$ with $\mathcal{M}_K$. The integers associated with the points in this set will be the same for points where one of the coordinates is negative enough. This number will be the number of transverse representatives with $\self=n$.

\section{Bypasses on Seifert surfaces}
In this section, we prove Theorem~\ref{maindestab} that says for a Legendrian knot $L$ in a contact manifold $(M,\xi)$ bounding a surface $\Sigma$ with $\tb(L)<-1$ and $\pm \rot(L)<\chi(\Sigma)$, we may isotopy $\Sigma$, rel $L$, so that there is a $\pm$-bypass for $L$ on $\Sigma$. 
\begin{proof}[Proof of Theorem~\ref{maindestab}]
For convenience in our arguments below, we assume that $L$ is not a Legendrian unknot. The case of a Legendrian unknot is easily dealt with as they have been classified \cite{EliashbergFraser09}. We will also assume that the complement of $L$ is tight, since otherwise, we can find a bypass along any arc \cite[Proposition~3.22]{Vogel09}, and we would be done. 

Since $\tb(L)<0$, we may isotop $\Sigma$ so that it is convex. It is well-known, see \cite{EtnyreHonda01b,Kanda1998}, that 
\[
\rot(L)=\chi(\Sigma_+)-\chi(\Sigma_-),
\]
where $\Sigma_+$ is the positive part of $\Sigma$ minus the dividing set $\Gamma$, and $\Sigma_-$ is the negative part. We note that Giroux's tightness criterion \cite[Theorem~3.5]{Honda00a} tells us that the assumption that the complement of $L$ is tight says that any disk component of $\Sigma_\pm$ has boundary containing arcs in $\Gamma$ and arcs in $\partial \Sigma$. 

We assume that $\rot(L)<\chi(\Sigma)$ (the $\rot(L)>-\chi(\Sigma)$ case is analogous). We will show that there is an arc in $\Gamma$ that cobounds a disk with $\partial \Sigma$ that is part of $\Sigma_-$. It is well-known, see \cite[Proposition~3.18]{Honda00a}, that this implies one may further isotop $\Sigma$ so that there is a negative bypass for $L$ on $\Sigma$. 


Notice that if $a$ is the number of arcs in $\Gamma$ then
\[
\chi(\Sigma)=\chi(\Sigma_+)+\chi(\Sigma_-) - a.
\]
Now let $d$ be the number of disk components of $\Sigma_-$ and $e$ be the sum of the Euler characteristics of the non-disk components of $\Sigma_-$. Thus
\[
\chi(\Sigma_+)=\chi(\Sigma) + a - d - e,
\]
and hence
\[
\rot(L)=\chi(\Sigma_+)-\chi(\Sigma_-)=\chi(\Sigma) - 2e + (a - 2d).
\]

Suppose that no disk in $\Sigma_-$ has boundary a single arc in $\Gamma$ and a single arc in $\partial \Sigma$. Thus, each disk component of $\Sigma_-$ contains at least $2$ arcs, and we see that $a-2d$ is non-negative. But of course $-2e$ is also non-negative, so $\rot(L)\geq \chi(\Sigma)$. This contradicts our assumption on the rotation number, and so we see that $\Sigma_-$ must contain a disk with boundary the union of one arc in $\Gamma$ and one arc in $\partial \Sigma$, as claimed. 
%
\end{proof}

\section{Destabilizing non-loose knots}
In this section, we prove Theorem~\ref{mainnonloose} which says that for a non-loose null-homologous Legendrian knot $L$ in an overtwisted contact manifold $(M,\xi)$, if $\pm \rot(L)<\chi(L)$ and $\tb(L)>0$, then there are non-loose knots $L_i$, for $i\in \N$, such that $L_0=L$ and $L_{i-1}=S_\pm(L_i)$. 
\begin{proof}[Proof of Theorem~\ref{mainnonloose}]
Suppose $L$ is such a non-loose Legendrian knot with $\rot(L)>-\chi(L)$ and $\tb(L)=n>0$ (the other case has a similar proof). Let $N_L$ be a standard neighborhood of $L$ and arrange that $\partial N_L$ is in standard form with ruling curves of slope $0$. As the ruling curves have a slope less than the dividing curves, each ruling curve is isotopic to a stabilization of $L$ with the same rotation number as $L$, \cite[Lemma~5.7]{DaltonEtnyreTraynor24}; moreover, since the ruling curves intersect the dividing curves non-trivially and the framing they get from the $\partial N_L$ and from a Seifert surface agree (because they have slope $0$) we know that their Thurston-Bennequin invariant is negative. Thus, if $\Sigma$ is a surface with $\partial \Sigma$ a ruling curve, then we can apply Proposition~\ref{maindestab} to find a bypass for $\partial N_L$ along the ruling curve. Attaching the bypass will result in a new solid torus $N$ containing $N_L$ with dividing slope $\infty$ (recall attaching a bypass to a torus with dividing slope $n>0$ along a ruling curve of slope $0$ will result in a convex torus with dividing slope given by the number clockwise of $n$ and anti-clockwise of $0$ in the Farey tesselation that is closest to $0$ and has an edge to $n$, and this is $\infty$). 

Now we note that inside $N\setminus N_L$ we can find convex tori $T_i$ for all integers $i>n$ with two dividing curves of slope $i$. Let $N_m$ be the solid tori that the $T_i$ bound. The ordering of the $T_i$ show that $N_i\subset N_{i'}$ whenever $i\leq i'$. By Corollary~\ref{detleg} we know the $N_i$ are standard neighborhoods of Legendrian knots $L_i$. We can take $N_n=N_L$ and so $L=L_n$. Now, since each $N_{i+1}\setminus N_i$ is a basic slice, we see that $L_i$ is a stabilization of $L_{i+1}$. Moreover, since the original basic slice was positive, we see that $N_{i+1}\setminus N_i$ is a positive basic slice. Thus $\rot(L_{i+1})-\rot(L_i)=1$ and we see that $L_i$ is a negative stablization of $L_{i+1}$. 
\end{proof}

\section{Transverse knots with fixed self-linking number}
In this section, we prove Theorem~\ref{detablizetransverse}, which says that for a null-homologous smooth knot type $K$ in a tight contact manifold $(M,\xi)$, there are only finitely many transverse representatives of $K$ with a fixed self-linking number. 
\begin{proof}[Proof of Theorem~\ref{detablizetransverse}]
Recall from Section~\ref{approx} that the classification of transverse knots in a knot type $K$ is the same as the negative stable classification of Legendrian knots. So one can understand transverse knots by understanding highly negatively stablized Legendrian knots. From Theorem~\ref{maindestab} we see that any transverse knot has a Legendrian approximation with rotation number $\chi(\Sigma)$ (where $\Sigma$ is a minimal genus Seifert surface for $K$). Given this, the self-linking number of a transverse realization of $K$ is determined by the Thurston-Bennequin invariant of a Legendrian approximation with rotation number $\chi(\Sigma)$. As noted in Section~\ref{destabsection}, the result of Colin, Giroux, and Honda \cite{ColinGirouxHonda09} says that there are only finitely many Legendrian knots with a fixed Thurston-Bennequin invariant and rotation number. Thus, there can only be finitely many transverse knots with a fixed self-linking number. 
\end{proof}

\section{Neighborhoods of transverse knots}
This section is devoted to the proof of Theorem~\ref{nbhdoftransverse} that says given a contact $3$-manifold  $(M,\xi)$ and a null-homologous knot type $K$ in a manifold $M$ and an integer $n$, there is a rational number $r$ such that any standard neighborhood with boundary slope $s<r$ of any transverse knot $T$ in the knot type $K$ with $\self(T)\geq n$ is unique up to contact isotopy. 

 To prove this theorem we first recall the well-known analogous result for Legendrian knots. 
\begin{lemma}\label{uniqueLegneigh}
Let $(M,\xi)$ be a contact manifold. Given two isotopic Legendrianian knots $L_0$ and $L_1$ and a standard neighborhood $N_i$ of $L_i$ for which the characteristic foliations on $\partial N_0$ and $\partial N_1$ agree, there is an ambient contact isotopy of $(M,\xi)$ taking $N_0$ to $N_1$. 
\end{lemma}
This is implicit in \cite{EtnyreHonda01b, Honda00a} as well as other papers, and a proof can be found in \cite{EtnyreTosunPre24}. The lemma is easy to prove. Specifically, one first observes that two Legendrian isotopic knots are ambiently contact isotopic \cite{Eliashberg1993, Etnyre05} and then that a fixed Legendrian knot has a canonical standard neighborhood which may be proven using the ideas in \cite{Honda00a}. 

Another ingredient we need concerns the contactomorphism of standard neighborhoods of Legendrian knots and a useful corollary of that. 
\begin{lemma}\label{lem1}
Consider the contact manifold $(S^1\times D^2, \xi)$ with convex boundary having two longitudinal dividing curves. Any contactomorphism of $(S^1\times D^2, \xi)$ is contact isotopic to the identity.
\end{lemma}
The proof of this lemma follows from \cite[Section~2.2]{EtnyreVertesi18}, see also \cite[Theorem~2.36]{Vogel2016}. 
A simple corollary of this lemma is the following that will be used in the proofs below.
\begin{corollary}\label{detleg}
A solid torus in a contact manifold $(M,\xi)$ with convex boundary having two longitudinal dividing curves is the standard neighborhood of a unique Legendrian knot. 
\end{corollary}


We now establish some preliminary results for the proof of Theorem~\ref{nbhdoftransverse}.
\begin{lemma}\label{mainhelper}
Suppose that the two standard neighborhoods $N_s$ and $N'_s$ of transversely isotopic transverse knots $T$ and $T'$  contained in a standard neighborhood of Legendrian knot approximations $L$ and $L'$ that are Legendrian isotopic, then $N_s$ is ambient contact isotopic to $N_s'$. 
\end{lemma}
\begin{proof}
Suppose $T$ and $T'$ are transversely isotopic and have standard neighborhoods $N_s$ and $N'_s$, respectively, and we know that there are Legendrian knots $L$ and $L'$ such that $N_s$ is contained in the standard neighborhood $N_L$ of $L$ and similarly for $N'_s$ and $N_{L'}$. By hypothesis, $L$ and $L'$ are Legendrian isotopic. Thus, by Lemma~\ref{uniqueLegneigh}, we know that there is an ambient contact isotopy taking $N_L$ to $N_{L'}$.  Thus, we can assume that $N_s$ and $N'_s$ are both contained in $N_L$. From the classification of tight contact structures on solid tori \cite{Giroux00, Honda00a}, there is a contactomorphism from $N_L$ to itself, that is the identity near the boundary, taking $N_s$ to $N'_s$, but it is known that such a contactomorphism of $N_L$ is contact isotopic to the identity by Lemma~\ref{lem1}. Thus, we have constructed the contact isotopy taking $N_s$ to $N'_s$. (If one wanted to, one could assume that under this isotopy, $T$ is taken to $T'$ as well.)
\end{proof}

\begin{lemma}\label{part1}
Let $K$ be a null-homologous knot type in a contact manifold $(M,\xi)$. If there exist integers $n$ and $m$ such that any transverse representatives $T$ of $K$ with $\self(T)=n$ has a unique Legendrian approximation $L_{m'}$ with $\tb(L_{m'})=m'$ for any $m'\leq m$, then any such $T$ has a well-defined standard neighborhood with any slope $s<\min\{m, -2g(K)+1+n\}$. 
\end{lemma}
\begin{proof}
By Lemma~\ref{mainhelper}, the proof will be complete if we can show that given a standard neighborhood $N_s$ of a transverse knot $T$ in the knot type $K$ with $\self(T)=n$ and $s<\min\{m, -2g(K)+1+n\}$, then $N_s$ is contained in the standard neighborhood $N_L$ of a Legendrian representative of $K$ with $\tb(L)=\lceil s\rceil$, the least integer greater than $s$. 

To see that this is true, given $N_s$, we know $\partial N_s$ has a standard neighborhood $T^2\times [-1,1]$ where the contact planes are tangent to the $[-1,1]$ factor and rotating in a left-handed manner as the interval is traversed. From this we see that $N_s$ is contained in $N_{s'}$ with $s'>s$ and we can assume that $s'<\lceil s\rceil$. We can perturb $\partial N_{s'}$ to be convex with two dividing curves of slope $s'$ and we can assume that $N_s$ is still in the interior of $N_{s'}$. Put $\partial N_{s'}$ in standard form with ruling curves of slope $0$, that is a ruling curve bounds a surface $\Sigma$ in the complement of $N_{s'}$. 

We now claim that $\rot(\partial \Sigma)<-\chi(\Sigma)$. To see this, notice that inside $N_s\subset N_{s'}$ there is a standard neighborhood $N_{\lfloor s\rfloor}$ of $T$. We can perturb $\partial N_{\lfloor s\rfloor}$ to be convex with two dividing curves of slope ${\lfloor s\rfloor}$. By Corollary~\ref{detleg} this will be the standard neighborhood of a Legendrian knot $L'$ whose transverse push-off is $T$. Moreover, we know that $\tb(L')={\lfloor s\rfloor}<-2g(K)+1+n$ and thus $\rot(L')=\tb(L')-\self(T)<-2g(K)+1$. Now the contact structure on $N_{s'}\setminus N_{\lfloor s\rfloor}$ is a minimally twisting contact structure on $T^2\times [0,1]$ that is contained in a standard neighborhood of a transverse knot. This means that the contact structure is universally tight and, moreover, if $A$ is a longitudinal annulus with one boundary component a ruling curve on $\partial N_{s'}$ and the other a ruling curve on $\partial N_{\lfloor s\rfloor}$, then the Euler class of the contact structure evaluated on this annulus (orieted as $L'\times [0,1]$) is positive due to the fact that basic slices inside standard neighborhoods of transverse knots are negative, see \cite[Section~3.6]{EtnyreHonda01b}. Thus we see that $\rot(L')-\rot(\partial \Sigma)>0$ and hence $\rot(\partial \Sigma)<-2g(K)-1= \chi(\Sigma)$ as claimed. 

Now Proposition~\ref{maindestab} gives is a bypass for $\partial N_{s'}$. Attaching the bypass will give a new solid torus with dividing slope larger than $s'$. We can continue to attach bypasses until we obtain a solid torus with convex bounday having $2$ dividing curves of slope $\lceil s\rceil$. This will be the standard neighborhood of a Legendrian knot $L$ with transverse push-off $T$ and this solid torus contains $N_s$ by construction. 
\end{proof}

As the next step in the proof of Theorem~\ref{nbhdoftransverse} we establish Lemma~\ref{part2}. Recall this lemma says that give a null-homologous knot type $K$ in a contact manifold $(M,\xi)$ and an integer $n$, there is an integer $m$ such that any two Legendrian approximations of a transverse knot $T$ in the knot type $K$ with $\self(T)>n$ with the same Thurston-Bennequin invariant less than $m$ are Legendrian isotopic. 

\begin{proof}[Proof of Lemma~\ref{part2}]
We begin by assuming that $\xi$ is either tight or if $\xi$ is overtwisted, then we consider non-loose transverse knots. 
We fix a knot type $K$. We will prove the lemma for transverse representatives of $K$ with $\self(K)=n$, since the self-linking number of transverse representatives we are considering is bounded above by $-\chi(K)$ this will establish the lemma.

By Proposition~\ref{maindestab} we know that there is a negative integer $b$ such that if $L$ is a Legendrian knot in the knot type $K$ and $\rot(L)<b$, then $L$ destabilizes. Now if $L$ is a Legendrian approximation of $T$ then $\tb(L)=\rot(L)+n$. So if $\tb(L)<b+n$, then $L$ destablizes. Thus all transverse representatives of $K$ with $\self=n$ have Legendrian approximations with $\tb=b+n$ and $\rot=b$. 

By the work of Colin, Giroux, and Honda \cite{ColinGirouxHonda09}, there are a finite number of Legendrian knots (in tight contact structures or non-loose in overtwisted contact structures with no Giroux torsion in their complements) with these fixed $\tb$ and $\rot$. So in $\xi$ there are a finite number of Legendrian representatives of $K$ with $\tb=b+n$ and $\rot=b$. If any two of them become isotopic after negative stabilization, then this will happen after a finite number of stabilizations. Thus there is some $k$ such that for distinct Legendrian approximations of transverse knots with $\self=n$ in $(M,\xi)$ with $\tb<b+n-k$ and $\rot=\tb-n$ remain distinct after any number of negative stabilizations. This means that their transverse push-offs are distinct and hence any transverse representative of $K$ with self-linking number $n$ has a unique Legendrian approximation with any $\tb<b+n-k$.
\end{proof}

We are now ready for the proof of one of our main results. 
\begin{proof}[Proof of Theorem~\ref{nbhdoftransverse}]
The theorem is an immediate follows from Lemmas~\ref{part1} and~\ref{part2}.
\end{proof}

\section{Examples of standard neighborhoods of transverse knots}
In this section we will establish our estimates on the size of standard neighborhoods of transverse knots that are well-defined up to ambient contact isotopy and explore examples of unique and non-unique standard neighborhoods of transverse knots. 
\subsection{General results on the size of well-defined neighborhoods of transverse knots}\label{genest}
We begin with the proof of Theorem~\ref{main2}. We recall that the theorem says that
if $K$ is a Legendrian simple knot in a tight contact manifold $(M,\xi)$ and $K=\partial \Sigma$ for some embedded surface $\Sigma$, then a standard neighborhood $N_s$ of a transverse knot $T$ in the knot type $K$ is unique up to contact isotopy if $s<\self(T)+\chi(\Sigma)$.
\begin{proof}[Proof of Theorem~\ref{main2}]
Let $N_s$ be a standard neighborhood of $T$ with $s<\self(T)+\chi(\Sigma)$. The argument in the last two paragraphs of the proof of Lemma~\ref{part1} shows that $N_s$ is contained in the standard neighborhood $N_L$ of a Legendrian realization of $K$ with $\tb(L)=\self(T)+\chi(\Sigma)$. The proof is now complete by Corollary~\ref {detleg} because we know that there is a unique Legendrian representative of $K$ with $\tb(L)=\self(T)+\chi(\Sigma)$ and $\rot(L)=\chi(\Sigma)$. 
\end{proof}

We now turn to the proof of Theorem~\ref{main3}, which says that 
if $K$ is Legendrian simple and uniformly thick, then any standard neighborhood of a transverse representative of $K$ is unique up to contact isotopy. 
\begin{proof}[Proof of Theorem~\ref{main3}]
Since every solid torus thickens to a neighborhood of a Legendrian representative of $K$ with maximal Thurston-Bennequin invariant, the result follows from Lemma~\ref{mainhelper}.
\end{proof}

We now address the possible standard neighborhoods of transverse knots in uniformly thick and Legendrian simple knot types. 
\begin{remark}\label{thealgorithm}
Suppose that $K$ is Legendrian simple and uniformly thick. From Theorem~\ref{main3} we know that neighborhoods of transverse representatives of $K$ are unique up to contact isotopy. To completely understand transverse neighborhoods, we are left to understand for which values of $s$ a transverse knot $T$ in the knot type $K$ has a standard neighborhood $N_s$. To this end, we consider the Legendrian mountain range for $K$ and the diagonal line $\self(T)=t-r$. Suppose that $t_T$ is the largest value of $t$ on this line. Then the correspondence between transverse knots and Legendrian approximations discussed in Section~\ref{approx} (coupled with uniform thickness) tells us that $T$ has standard neighrohoods for any $s<t_T$.
\end{remark}

\subsection{Examples of standard neighborhoods of transverse knots}\label{examp}
We will now consider standard neighborhoods of transverse knots in specific knot types. We begin with the $(2,-2n-1)$-torus knot $K$. We recall that $K$ is Legendrian simple, and the Legendrian classification is given by the mountain range in Figure~\ref{fig:mrneg22kp1}. See \cite{EtnyreHonda01b}.
 \begin{figure}[htb]{
\begin{overpic}
{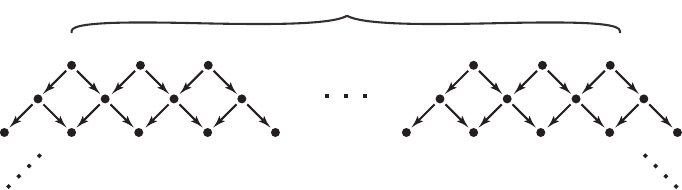}
\put(163, 88){$2n$}
\put(17, 65){$-2n+1$}
\put(283, 65){$2n-1$}
\put(-43, 58){$-4n-2$}
\put(-43, 42){$-4n-3$}
\put(-43, 26){$-4n-4$}
\end{overpic}}
\caption{The mountain range for the positive $(2,-2n-1))$-torus knots.}
\label{fig:mrneg22kp1}
\end{figure}
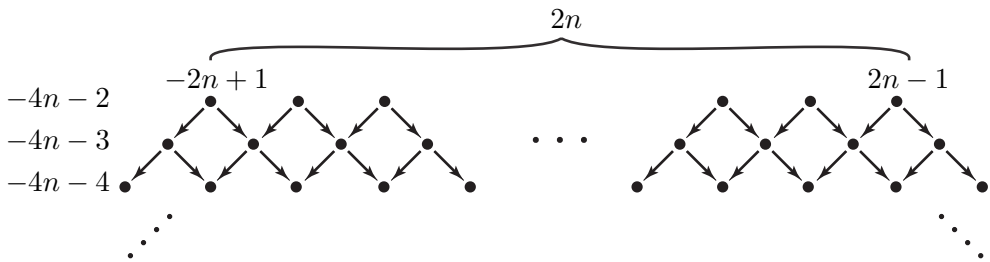   
 
\begin{proof}[Proof of Theorem~\ref{ex1}]
Let $K$ be the $(2, -2n-1)$-torus knot. We recall that $K$ is uniformly thick \cite{EtnyreHonda05}. Thus by Theorem~\ref{main3} we know that every standard neighborhood of a transverse representative of $K$ is determined up to contact isotopy by the slope of the characteristic foliation on its boundary.  We are left to see what slopes can be realized by such neighborhoods. 

Suppose $T$ is a transverse representative of $K$ with $\self(T)=l \in \{-2n-3, -2n-5, \ldots,-6n-1\}$. Then we consider the line $t=r+l$ in the mountain range of $K$ shown in Figure~\ref{fig:mrneg22kp1}. (Recall our discussion at the end of Section~\ref{approx} about how to relate transverse knots to the mountain range of Legendrian knots.) We notice that for points in this set, the largest second coordinate is $-4n-2$. 
Now, by Lemma~\ref{subtori} we see that we can find standard neighborhoods $N_s$ of $T$ with any $s<-4n-2$. Now, suppose that there is a standard neighborhood $N_s$ of $T$ with $s\geq -4n-2$. If the slope is $-4n-2$, then in a neighborhood of $\partial N_s$ we can find a torus $T$ with linear foliation larger than $s$ and bounding a solid torus containing $N_s$. Thus, we can assume that $s>-4n-2$. By the uniform thickness hypothesis, we know that $N_s$ is contained in a solid torus $S$ that is a neighborhood of a Legendrian representative of $K$. This neighborhood has convex boundary with two dividing curves of slope $-4n-2$, which is less than $s$. Thus, $S\setminus N_s$ is a thickened torus, and we can perturb the back face to be convex with slope $s$. The classification of tight contact structures on such a thickened torus \cite{Giroux00, Honda00a} implies that one can find a convex torus parallel to the boundary with any dividing slope in $[s, -4n-2]$ (here we are thinking of this as an interval of slopes on the Farey graph, for the conventions we are using see \cite{EtnyreRoy21}). Since this interval contains $\infty$ we see that there is a meridional disk in $S$ with its boundary a Legendrian divide, and hence this is an overtwisted disk. Thus, such an $N_s$ could not exist. 

Now if $T$  is a transverse representative of $K$ with $\self(T)=2k-1$, for $k>3n$, then the same argument as above shows that $T$ has a standard neighborhood $N_s$ if and only if $s<-k-n-2$. 
\end{proof}

We now turn to the $(4,-9)$-torus knot $K$. We know that $K$ is Legendrian simple \cite{EtnyreHonda01b} and uniformly thick \cite{EtnyreHonda05}. The classification of Legendrian representatives of $K$ is given in Figure~\ref{fig:49mountain}. 
\begin{figure}[htb]{
\begin{overpic}
{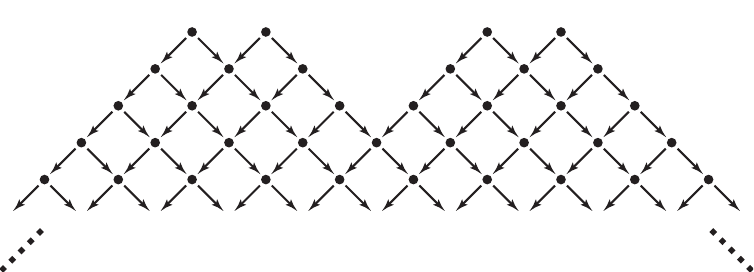}
\put(178, 125){$0$}
\put(196, 125){$1$}
\put(214, 125){$2$}
\put(232, 125){$3$}
\put(250, 125){$4$}
\put(268, 125){$5$}
\put(286, 125){$6$}
\put(304, 125){$7$}
\put(322, 125){$8$}
\put(340, 125){$9$}
\put(155, 125){$-1$}
\put(137, 125){$-2$}
\put(119, 125){$-3$}
\put(101, 125){$-4$}
\put(83, 125){$-5$}
\put(65, 125){$-6$}
\put(47, 125){$-7$}
\put(29, 125){$-8$}
\put(11, 125){$-9$}
\put(-15, 113){$-36$}
\put(-15, 95){$-37$}
 \put(-15, 77){$-38$}
 \put(-15, 59){$-39$}
 \put(-15, 41){$-40$}
\end{overpic}}
\caption{The image of the map $\Psi$ for the $(4,-9)$-torus knot..}
\label{fig:49mountain}
\end{figure}
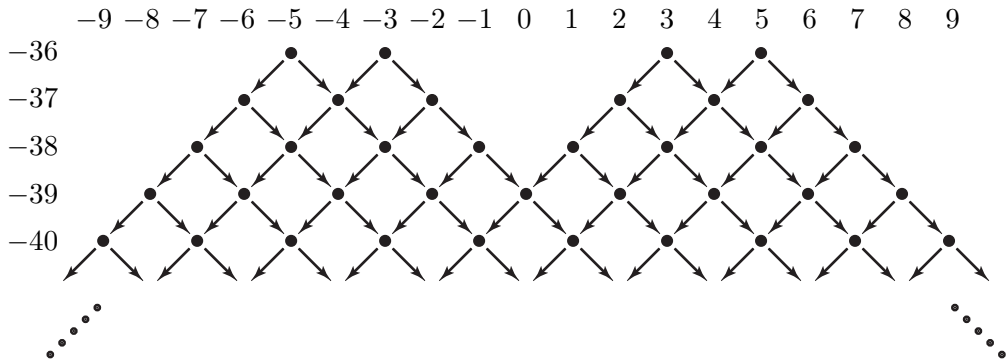
We are now ready to prove Theorem~\ref{ex2}.
\begin{proof}[Proof of Theorem~\ref{ex2}]
The proof is essentially the same as the proof of Theorem~\ref{ex1} using Figure~\ref{fig:49mountain} instead of Figure~\ref{fig:mrneg22kp1}.
\end{proof}

We now consider the unknot, which is not uniformly thick. 
\begin{proof}[Proof of Teorem~\ref{ex3}]
There is a unique Legendrian unknot $L$ with $\tb(L)=-1$ and $\rot(L)=0$, and all other Legendrian unknots are stabilizations of this one. We claim that any solid torus $S$ in the knot type of the unknot with boundary having a linear foliation of slope $s<-1$ is contained in a solid torus that is a standard neighborhood of $L$. Given this, we can use the proof of Theorem~\ref{ex1} to see that if $T$ is a transverse unknot and $\self(T)=-2k-1$ for $k>0$, then $T$ has a standard neighborhood $N_s$ if and only if $s<-k-1$, and the standard neighborhood of any such transverse knot is determined up to contact isotopy by the slope on its boundary. 

To see that our claim is true, we suppose $S$ is a solid torus in the knot type of the unknot with boundary having a linear foliation of slope $s<-1$. Then $S'=S^3\setminus S$ is a solid torus. Suppose that $s$ is not an integer for the moment. Then there is a torus $T'$ inside $S'$ and isotopic to $\partial S'$ that is convex with two dividing curves of slope $n= \lceil s\rceil$. This $T$ bounds a solid torus $S_n$ that contains $S$; moreover, by Corollary~\ref{detleg}, $S_n$ is a standard neighborhood of a Legendrian unknot $L'$, and since $L'$ destabilizes to $L$, we know that $S$ is contained in a standard neighborhood of $L$ as claimed. 

Now if $s<-1$ is an integer, then as we have argued above, we can consider a standard neighborhood of $\partial N_s$ to find a new standard neighborhood $N_{s'}$ containing $N_s$ with $s'>s$ not an integer. So this case follows from the non-integer case above. 

Now suppose $T$ is a transverse unknot with $\self(T)=-1$. Since $T$ is the transverse push-off of $L$, we see that $T$ will have standard neighborhoods $N_s$ for any $s<-1$, and they will be unique as above. We are left to see that $T$ also has standard neighborhoods $N_s$ with $s\in[-1,0)$, that they are unique, and that there can be no standard neighborhood with $s\geq 0$. To this end, we recall a construction of the standard contact structure on $S^3$. Consider the manifold $T^2\times [0,1]$ with the contact structure $\ker\left(\cos \frac \pi 2 t\, d\theta + \sin \frac \pi 2 t\, d\phi\right)$, where $\theta$ and $\phi$ are angular coordinates on $T^2$ and $t$ is the coordinate on $[0,1]$. The characteristic foliation on $T^2\times \{0\}$ is by vertical curves (constant $\theta$ curves) and on $T^2\times \{1\}$ by horizontal curves. It is clear that $S^3$ is $T^2\times[0,1]$ with the leaves of the characteristic foliations on $T^2\times\{0\}$ and $T^2\times\{1\}$ collapsed to points (this is just $S^3$ expressed as the join of $S^1$ and $S^1$). Moreover, performing a contact cut \cite{Lerman01} on $T^2\times\{0\}$ and $T^2\times\{1\}$ gives a contact structure on $S^3$. This is the standard contact structure $\xi_{std}$ on $S^3$. The image of $T^2\times\{0\}$ under this collapse is a transverse unknot, and it is easy to see that it bounds a disk containing a single critical point. This implies that $\self=-1$ and thus is $T$. Notice that $T$ has neighborhoods $T^2\times[0, a]$ with a contact cut performed on $T^2\times\{0\}$ and these are standard neighborhoods of $T$ with boundary having linear foliations ranging in $(-\infty, 0)$. 

We know that any standard neighborhood $N_s$ of $T$ with $s<-1$ is unique up to contact isotopy, so now consider the case when $s\in[-1,0)$. If $s\not = -1/n$ for some positive $n$, then let $n$ be the smallest integer such that $s<-1/n$. Considering a neighborhood of $\partial N_s$ we see that $N_s$ is contained in $N_{s'}$ where $s'\in(s,-1/n)$. We can make the boundary of $N_{s'}$ convex with two dividing curves of slope $s'$ and still have this torus contain $N_s$. We can now make the ruling curves on $N_{s'}$ have slope zero. Take a disk in the complement of $N_{s'}$ with boundary a ruling curve. We can make this disk convex, and as it will have more than one dividing curve (since $s'$ is not one over an integer), we can find a bypass for $N_{s'}$. Attaching the bypass, we can find a new torus containing $N_{s'}$ with larger dividing slope. We can continue until $N_{s}$ is contained in a solid torus $N$ with convex boundary having dividing slope $-1/n$. Inside $N_s$ we can find a solid torus $N'$ with convex boundary and dividing slope $-1/(n-1)$. Notice that $N\setminus N'$ is a thickened torus that is a basic slice. Since it is the union of the two thickened tori $N_s\setminus N'$ and $N\setminus N_s$ and the contact structure on $N_s\setminus N'$ is determined (as a subset of a standard neighborhood of a transverse knot), we know the contact structure on $N\setminus N_s$ is uniquely determined (that is there is one possible tight contact structure on this that can be glued to $N_s\setminus N'$ to obtain a tight contact structure). Note also that $S^3\setminus N$ is a solid torus that supports a unique tight contact structure (since its dividing curves are longitudinal with respect to the meridional disk in $S^3\setminus N$). Thus the tight contact structure on $S^3\setminus N_s$ is uniquely determined by $N_s$. So if $N_s$ and $N'_s$ are two standard neighborhoods of $T$, then there is a contactomorphism of $S^3$ taking $N_s$ to $N'_s$. But since the contactomorphism group of $(S^3,\xi_{std})$ is path-connected \cite{Eliashberg92a} $N_s$ is ambient contact isotopic to $N'_s$. 

We end by showing that $T$ cannot have a standard neighborhood $N_s$ with $s\geq 0$. If such a neighborhood existed, then $T$ would have a standard neighborhood $N_0$ as discussed in Section~\ref{nbhds}. But a leaf of the characteristic foliation on $\partial N_0$ bounds a disk and this disk will be an overtwisted disk.
\end{proof}

We now move on to consider transverse knots that do not have unique standard neighborhoods. Specifically, we consider the right-handed trefoil in $(S^3,\xi_{std})$ and prove Theorem~\ref{ex4}.
\begin{proof}[Proof of Theorem~\ref{ex4}]
The right-handed trefoil $K$ is Legendrian simple \cite{EtnyreHonda01b}, and any solid torus in the knot type of the right-handed trefoil with convex boundary having dividing curves of slope less than $0$ can be contained in a standard neighborhood of a Legendrian knot with $\tb=0$, see \cite[Section~3.1]{EtnyreHonda05}. Thus, arguing as in the proof of Theorem~\ref{ex1} we can see that a transverse representative of $K$ with self-linking number $-2k-1$, for $k\geq 0$, has a standard neighborhood $N_s$ if and only if $s<-k$. Moreover, any standard neighborhood of any such transverse representative of $K$ is determined up to contact isotopy by the slope on its boundary. 

We now consider the transverse right-handed trefoil $T$ with $\self=1$. This is the transverse push-off of the Legendrian right-handed trefoil $L$ with $\tb=1$ and $\rot=0$. Thus, we see, by Lemma~\ref{subtori}, that $T$ will have standard neighborhoods $N_s$ for any $s<1$. In \cite{EtnyreHonda05} it was shown that there are no solid tori in the knot type of $K$ with convex boundary having dividing curves of slope greater than $1$. Any standard neighborhood $N_s$ of $T$ with $s\geq 1$ would contradict that. (This is clear for $s>1$, but for $s=1$ we note that a standard neighborhood of $\partial N_1$ would show that $N_1$ is contained in $N_s$ for some $s>1$ so is true for $s=1$ too.) So we know the possible values of $s$ for which $T$ has a standard neighborhood $N_s$. 

Our discussion above shows that any neighborhood $N_s$ of $T$ with $s<0$ is unique up to contact isotopy. To understand the other standard neighborhoods we now recall the classification of solid tori in the knot type of the right-handed trefoil from \cite{EtnyreLafountainTosun12}. There are tori $S_n^\pm$ for integers $n>1$ and $S_1$ in the knot type $K$, where $N_n^\pm$ has convex boundary with two dividing curves of slope $1/n$ and $S_1$ has convex boundary with two dividing curves of slope $1$. All of these tori are distinct up to contact isotopy, any solid torus $S$ in the knot type $K$ is contained in one of these tori, and all of these tori are non-thickenable in the sense that if $S$ is a solid torus with convex boundary containing $S_n^\pm$ then the dividing slope of $\partial S$ is $1/n$ (and similarly for $S_1$). Moreover, if $S$ is a solid torus with convex boundary having dividing slope larger than $0$ and contained in $S_n^\pm$ then $S$ is not contained in any other $S_m^\pm$, but if $\partial S$ has dividing slope less than or equal to $0$ then it can ``thicken'' to $S_1$. The $S_n^-$ and $S_1$ are neighborhoods of $T$ with convex boundary. So if $N_s$ is a standard neighborhood with $s\in[1/(n+1),1/n)$ then it can be a sub-torus of any $S^-_m$ for $m=1,\ldots, n$, and these will all be distinct up to contact isotopy, but if two neighborhoods are in the same $S_m^-$ (with the same slope on the boundary) then they will be isotopic (by the same argument as at the end of the proof of Theorem~\ref{ex3}). 

We finally consider standard neighborhoods $S_0$ of $T$. To this end, we recall how the complement of the tori $S^\pm_n$ can be built. Let $M$ be the result of $0$ surgery on the right-handed torus. This is a torus bundle over the circle. We can think of $M$ as $T^2\times [0,1]$ with the two boundary components glued together by a diffeomorphism of the torus. In \cite{Giroux00, Honda00a} it was shown that $M$ admits a family of tight contact structure $\xi_n$. Each $\xi_n$ is tangent to the $[0,1]$-factor and $\xi_n$ has Giroux torsion $n-1$. We can now consider the Legendrian $L_n$ in $\xi_n$ that is $\{p\}\times [0,1]$ in $M$ where $p$ is a fixed point of the gluing map. The complement $C_n$ of a standard neighborhood of $L_n$ in $(M,\xi_n)$ is contactomorphic to the complement of $S^\pm_n$ in $(S^3,\xi_{std})$. Now consider the transverse push-off $T_n$ of $L_n$. We can consider the compactification $C'_n$ of $M-T_n$ obtained by adding a torus boundary. It is easy to see that the characteristic foliation on $\partial C'_n$ is linear and the leaves are meridians.  
Now consider $N_0^n$ be the standard neighborhood of $T$ in $S^3$ with boundary slope $0$ inside of $N_n^-$. Then $S^3\setminus N_0^n$ is contactomorphic to $C_n'$. We claim that none of the $C_n'$ are contactomorphic. This will show that all the $N_0^n$ are distinct as claimed. We note that $(M,\xi_n)$ is obtained from $C_n'$ by a contact cut (recall, this means collapsing the leaves of the characteristic foliation on $\partial C_n'$, \cite{Lerman01}). If $C_n'$ and $C_m'$ were contactomorphic, the contactomorphism would have to take the torus boundary to the torus boundary and fix the characteristic foliation. Thus, the contact cuts would be contactomorphic, but since none of the $\xi_n$ are contactomorphic, none of the $C_n'$ can be either. 

We finally note that a standard neighborhood $N_s$ with $s\in[1/2,1)$ then $N_s$ can only sit in $N_1$, the standard neighborhood of $L$, and thus it is unique, up to ambient contact isotopy. 
\end{proof}
We note that Theorem~\ref{curous} follows directly from the latter part of the last proof.
\begin{proof}[Proof of Theorem~\ref{curous}]
The contact manifold $C_n'$ in the previous proof are the contact structure claimed in the theorem. 
\end{proof}
\begin{remark}\label{remark}
We note that all positive torus knots have non-thickenable tori \cite{EtnyreLafountainTosun12}, so the arguments in the proof of Theorem~\ref{ex4} above will show that they admit transverse representatives that do not have unique standard neighborhoods. 
\end{remark}

\def\cprime{$'$}


\begin{thebibliography}{10}

\bibitem{ChristianMenkePre}
Austin Christian and Michael Menke.
\newblock Splitting symplectic fillings, 2019.

\bibitem{ColinGirouxHonda09}
Vincent Colin, Emmanuel Giroux, and Ko~Honda.
\newblock Finitude homotopique et isotopique des structures de contact tendues.
\newblock {\em Publ. Math. Inst. Hautes \'Etudes Sci.}, (109):245--293, 2009.

\bibitem{Conway19}
James Conway.
\newblock Transverse surgery on knots in contact 3-manifolds.
\newblock {\em Trans. Amer. Math. Soc.}, 372(3):1671--1707, 2019.

\bibitem{DaltonEtnyreTraynor24}
Jennifer Dalton, John~B. Etnyre, and Lisa Traynor.
\newblock Legendrian torus and cable links.
\newblock {\em J. Symplectic Geom.}, 22(1):11--108, 2024.

\bibitem{Eliashberg92a}
Yakov Eliashberg.
\newblock Contact {$3$}-manifolds twenty years since {J}. {M}artinet's work.
\newblock {\em Ann. Inst. Fourier (Grenoble)}, 42(1-2):165--192, 1992.

\bibitem{Eliashberg1993}
Yakov Eliashberg.
\newblock Legendrian and transversal knots in tight contact {$3$}-manifolds.
\newblock In {\em Topological methods in modern mathematics ({S}tony {B}rook,
  {NY}, 1991)}, pages 171--193. Publish or Perish, Houston, TX, 1993.

\bibitem{EliashbergFraser09}
Yakov Eliashberg and Maia Fraser.
\newblock Topologically trivial {L}egendrian knots.
\newblock {\em J. Symplectic Geom.}, 7(2):77--127, 2009.

\bibitem{EpsteinFuchsMeyer01}
Judith Epstein, Dmitry Fuchs, and Maike Meyer.
\newblock Chekanov-{E}liashberg invariants and transverse approximations of
  {L}egendrian knots.
\newblock {\em Pacific J. Math.}, 201(1):89--106, 2001.

\bibitem{EtnyreVertesi18}
John Etnyre and Vera V\'{e}rtesi.
\newblock Legendrian satellites.
\newblock {\em Int. Math. Res. Not. IMRN}, (23):7241--7304, 2018.

\bibitem{Etnyre05}
John~B. Etnyre.
\newblock Legendrian and transversal knots.
\newblock In {\em Handbook of knot theory}, pages 105--185. Elsevier B. V.,
  Amsterdam, 2005.

\bibitem{Etnyre13}
John~B. Etnyre.
\newblock On knots in overtwisted contact structures.
\newblock {\em Quantum Topol.}, 4(3):229--264, 2013.

\bibitem{EtnyreHonda01b}
John~B. Etnyre and Ko~Honda.
\newblock Knots and contact geometry. {I}. {T}orus knots and the figure eight
  knot.
\newblock {\em J. Symplectic Geom.}, 1(1):63--120, 2001.

\bibitem{EtnyreHonda05}
John~B. Etnyre and Ko~Honda.
\newblock Cabling and transverse simplicity.
\newblock {\em Ann. of Math. (2)}, 162(3):1305--1333, 2005.

\bibitem{EtnyreLafountainTosun12}
John~B. Etnyre, Douglas~J. LaFountain, and B{\"u}lent Tosun.
\newblock Legendrian and transverse cables of positive torus knots.
\newblock {\em Geom. Topol.}, 16(3):1639--1689, 2012.

\bibitem{EtnyreMinMukherjee22Pre}
John~B. Etnyre, Hyunki Min, and Anubhav Mukherjee.
\newblock Non-loose torus knots, 2022.

\bibitem{EtnyreMinTosunVarvarezosPre}
John~B. Etnyre, Hyunki Min, B{\"u}lent Tosun, and Konstantinos Varvarezos.
\newblock Tight surgeries on torus knots, 2026.

\bibitem{EtnyreNgVertesi13}
John~B. Etnyre, Lenhard~L. Ng, and Vera V{\'e}rtesi.
\newblock Legendrian and transverse twist knots.
\newblock {\em J. Eur. Math. Soc. (JEMS)}, 15(3):969--995, 2013.

\bibitem{EtnyreRoy21}
John~B. Etnyre and Agniva Roy.
\newblock Symplectic fillings and cobordisms of lens spaces.
\newblock {\em Trans. Amer. Math. Soc.}, 374(12):8813--8867, 2021.

\bibitem{EtnyreTosunPre24}
John~B. Etnyre and B{\"u}lent Tosun.
\newblock {\em Low dimensional contact geometry}.
\newblock To appear.

\bibitem{Gay02a}
David~T. Gay.
\newblock Symplectic 2-handles and transverse links.
\newblock {\em Trans. Amer. Math. Soc.}, 354(3):1027--1047 (electronic), 2002.

\bibitem{Geiges08}
Hansj{\"o}rg Geiges.
\newblock {\em An introduction to contact topology}, volume 109 of {\em
  Cambridge Studies in Advanced Mathematics}.
\newblock Cambridge University Press, Cambridge, 2008.

\bibitem{Giroux00}
Emmanuel Giroux.
\newblock Structures de contact en dimension trois et bifurcations des
  feuilletages de surfaces.
\newblock {\em Invent. Math.}, 141(3):615--689, 2000.

\bibitem{Guyard15}
Thomas Guyard.
\newblock {\em Sur le calcul d'invariants et l'engendrement des noeuds
  transverses dans les vari{\'e}t{\'e}s de contact de dimension trois}.
\newblock PhD thesis, Universit{\'e} de Nantes, 2015.

\bibitem{Honda00a}
Ko~Honda.
\newblock On the classification of tight contact structures. {I}.
\newblock {\em Geom. Topol.}, 4:309--368 (electronic), 2000.

\bibitem{Kanda1998}
Yutaka Kanda.
\newblock On the {T}hurston-{B}ennequin invariant of {L}egendrian knots and
  nonexactness of {B}ennequin's inequality.
\newblock {\em Invent. Math.}, 133(2):227--242, 1998.

\bibitem{Lerman01}
Eugene Lerman.
\newblock Contact cuts.
\newblock {\em Israel J. Math.}, 124:77--92, 2001.

\bibitem{Lutz77}
Robert Lutz.
\newblock Structures de contact sur les fibr\'es principaux en cercles de
  dimension trois.
\newblock {\em Ann. Inst. Fourier (Grenoble)}, 27(3):ix, 1--15, 1977.

\bibitem{Swiatkowski92}
Jacek \'{S}wi\k{a}tkowski.
\newblock On the isotopy of {L}egendrian knots.
\newblock {\em Ann. Global Anal. Geom.}, 10(3):195--207, 1992.

\bibitem{Vogel09}
Thomas Vogel.
\newblock Existence of {E}ngel structures.
\newblock {\em Ann. of Math. (2)}, 169(1):79--137, 2009.

\bibitem{Vogel2016}
Thomas Vogel.
\newblock On the uniqueness of the contact structure approximating a foliation.
\newblock {\em Geom. Topol.}, 20(5):2439--2573, 2016.

\end{thebibliography}
\end{document}